\RequirePackage{snapshot}
\RequirePackage{fix-cm}
\documentclass[smallcondensed,envcountsect,natbib]{svjour3}       %
\smartqed  %
\usepackage{amsmath,bbm,bm}
\usepackage{mathtools}
\usepackage{braket}
\usepackage{graphicx}
\usepackage{amsopn}
\usepackage[margin=1.75in,centering]{geometry}
\input{macros}
\newcommand{\conj}[1]{#1^{\star}}
\newcommand{\biconj}[1]{#1^{\star\star}}
\newcommand{\ts}{\tau^*}

\DeclareMathOperator{\cl}{cl}

\newcommand{\popt}{\tau^*_{\smash p}}
\newcommand{\dopt}{\tau^*_d}
\DeclareMathOperator{\vfepi}{vfepi}

\usepackage[svgnames]{xcolor} 
\usepackage{hyperref}  
\hypersetup{
  breaklinks,
  colorlinks,
  citecolor={DarkBlue},
  urlcolor={DarkBlue},
  linkcolor={DarkBlue}
}

\usepackage[capitalise,nameinlink]{cleveref}
\Crefname{figure}{Figure}{Figures}
\numberwithin{equation}{section}
\numberwithin{figure}{section}

\begin{document}
\def\makeheadbox{\relax}

\title{A perturbation view of level-set methods for\\convex optimization}

\titlerunning{Perturbation view of level-set methods}

\author{Ron Estrin\and
        Michael P.\@ Friedlander 
      }

\authorrunning{Estrin and Friedlander} %

\institute{R. Estrin \at
  Institute for Computational and Mathematical Engineering
  \\Stanford University, Stanford, CA, USA
  \\\email{restrin@stanford.edu}
  \and
  M. P. Friedlander \at         
  Department of Computer Science and Department of Mathematics
  \\University of British Columbia
  \\Vancouver, BC, V6R 1Y8, Canada
  \\\email{michael@friedlander.io}
}

\date{January 17, 2020 (revised May 15, 2020)}

\maketitle

\begin{abstract}
  Level-set methods for convex optimization are predicated on the idea that
  certain problems can be parameterized so that their solutions can be recovered
  as the limiting process of a root-finding procedure. This idea emerges time
  and again across a range of algorithms for convex problems. Here we
  demonstrate that strong duality is a necessary condition for the level-set
  approach to succeed. In the absence of strong duality, the level-set method
  identifies $\epsilon$-infeasible points that do not converge to a
  feasible point as $\epsilon$ tends to zero. The level-set approach
  is also used as a proof technique for establishing sufficient conditions for
  strong duality that are different from Slater's constraint qualification.
\end{abstract}

\section{Introduction}

Duality in convex optimization may be interpreted as a notion of sensitivity
of an optimization problem to perturbations of its data. Similar notions of
sensitivity appear in numerical analysis, where the effects of numerical errors
on the stability of the computed solution are of central concern. Indeed,
backward-error analysis \cite[\S1.5]{higham2002accuracy} describes the related
notion that computed approximate solutions may be considered as exact solutions
of perturbations of the original problem. It is natural, then, to ask if duality
can help us understand the behavior of a class of numerical algorithms for
convex optimization. In this paper, we describe how the level-set method
\citep{spgl12007,bf2008,2016aravkinlevel} produces an incorrect solution when
applied to a problem for which strong duality fails to hold. In other words, the
level-set method cannot succeed if there does not exist a dual pairing that is
tight. This failure of strong duality indicates that the stated optimization
problem is brittle, in the sense that its value as a function of small
perturbations to its data is discontinuous; this violates a vital assumption
needed for the level-set method to succeed.

Consider the convex optimization problem
\begin{equation}
  \label{eq:cvx_primal}
  \minimize{x \in \Xscr} \enspace f(x) \enspace\st\enspace g (x) \leq 0,
  \tag{\mbox{P}}
\end{equation}
where $f$ and $g$ are closed proper convex functions that map $\Real^n$ to the
extended real line $\Real\cup\{\infty\}$, and $\Xscr$ is a convex set in
$\Real^n$. Let the optimal value $\popt$ of \eqref{eq:cvx_primal} be finite,
which indicates that that~\eqref{eq:cvx_primal} is feasible. In the context of
level-set methods, we may think of the constraint $g(x)\le0$ as representing
a computational challenge. For example, there may not exist any efficient
algorithm to compute the projection onto the constraint set
$\set{x\in\Xscr|g(x)\le0}$. In many important cases, the objective function has
a useful structure that makes it computationally convenient to swap the roles of
the objective $f$ with the constraint $g$, and instead to solve the
\emph{level-set problem}
\begin{equation}
  \label{eq:cvx_flipped}
  \minimize{x \in \Xscr} \enspace g(x) \enspace\st\enspace f (x) \leq \tau,
  \tag{\mbox{Q$_\tau$}}
\end{equation}
where $\tau$ is an estimate of the optimal value $\popt$. The term ``level set''
points to the feasible set of problem \eqref{eq:cvx_flipped}, which is the $\tau$
level set $\set{x | f(x)\le\tau}$ of the function $f$.

If $\tau \approx \popt$, the level-set constraint $f(x)\leq\tau$ ensures that a
solution $x_\tau\in\Xscr$ of this problem causes $f(x_\tau)$ to have a value
near the optimal value $\popt$. If, additionally, $g(x_\tau)\le0$, then $x_\tau$
is a nearly optimal and feasible solution for~\eqref{eq:cvx_primal}. The
trade-off for this potentially more convenient problem is that we must compute a
sequence of parameters $\tau_k$ that converges to $\popt$.

\subsection{Objective and constraint reversals}

The technique of exchanging the roles of the objective and constraint functions
has a long history. For example, the isoperimetric problem, which dates back to
the second century B.C.E., seeks the maximum area that can be circumscribed by
a curve of fixed length \citep{Wiegert:2010}. The converse problem seeks the
minimum-length curve that encloses a certain area. Both problems yield the same
circular solution. The mean-variance model of financial portfolio optimization,
pioneered by ~\citet{Mark1987}, is another example. It can be phrased
as either the problem of allocating assets that minimize risk (i.e., variance)
subject to a specified mean return, or as the problem of maximizing the mean
return subject to a specified risk. The correct parameter choice, such as $\tau$ in the
case of the level-set problem \eqref{eq:cvx_flipped}, causes both problems to
have the same solution.

The idea of rephrasing an optimization problem as a root-finding problem
appears often in the optimization literature. The celebrated
Levenberg-Marquardt algorithm~\citep{Marq63,Morr60}, and trust-region
methods~\citep{conngoultoin:2000} more generally, use a root-finding
procedure to solve a parameterized version of the optimization
problem.
\cite{lemaneminest:1995} develop a root-finding procedure for a level-bundle
method for general convex optimization.
The widely used SPGL1 software package for sparse
optimization~\citep{BergFrie:2007b} implements the level-set method for
obtaining sparse solutions of linear least-squares and underdetermined
linear systems \citep{BergFriedlander:2008,BergFriedlander:2011}.

\subsection{Duality of the value function root}\label{sec:overview}

Define the optimal-value function, or simply the \emph{value
  function}, of~\eqref{eq:cvx_flipped} by
\begin{equation}
  \label{eq:Pf}
  v(\tau) = \inf_{x\in\Xscr}\set{g(x) | f(x) \le \tau}.
\end{equation}
If the constraint in~\eqref{eq:cvx_primal} is \emph{active} at a solution, that
is, $g(x)=0$, this definition then suggests that the optimal value $\popt$ of \eqref{eq:cvx_primal} is a
root of the equation
\begin{equation*}%
  v(\tau)=0,
\end{equation*}
and in particular, is the leftmost root:
\begin{equation} \label{eq:11}
  \popt = \inf\set{\tau | v(\tau) = 0}.
\end{equation}
The surprise is that this is not always true.

In fact, as we demonstrate in this paper, the failure of strong duality for
\eqref{eq:cvx_primal} implies that
\begin{equation}\label{eq:duality-gap}
  \dopt \coloneqq \inf\set{\tau | v(\tau) = 0} < \popt.
\end{equation}
Thus, a root-finding algorithm, such as bisection or Newton's method,
implemented so as to yield the leftmost root of the equation $v(\tau)=0$ will
converge to a value of $\tau$ that prevents~\eqref{eq:cvx_flipped} from attaining
a meaningful solution. This phenomenon is depicted in \cref{fig:val-func-sdp},
and is manifested by the semidefinite optimization problem in
\cref{sec:example_2}. Moreover, the infimal value
in~\eqref{eq:duality-gap}, defined here as $\dopt$, coincides with the optimal value of any dual pairing
of \eqref{eq:cvx_primal} that arises from Fenchel-Rockafellar convex duality
\cite[Theorem 11.39]{rtrw:1998}. These results are established by Theorems~\ref{thm:val-fnc-roots} and~\ref{thm:val-fnc-roots-general}.

We do not assume that our readers are experts in convex duality theory, and so
we present an abbreviated summary of the machinery needed to develop our main
results. We also describe a generalized version of the level-set pairing between
the problems~\eqref{eq:cvx_primal} and~\eqref{eq:cvx_flipped}, and thus
establish \cref{thm:val-fnc-roots-general}. We show in \cref{sec:examples} how
these theoretical results can be used to establish sufficient conditions for
strong duality.

\begin{figure}[t]
  \centering
  \includegraphics[page=1]{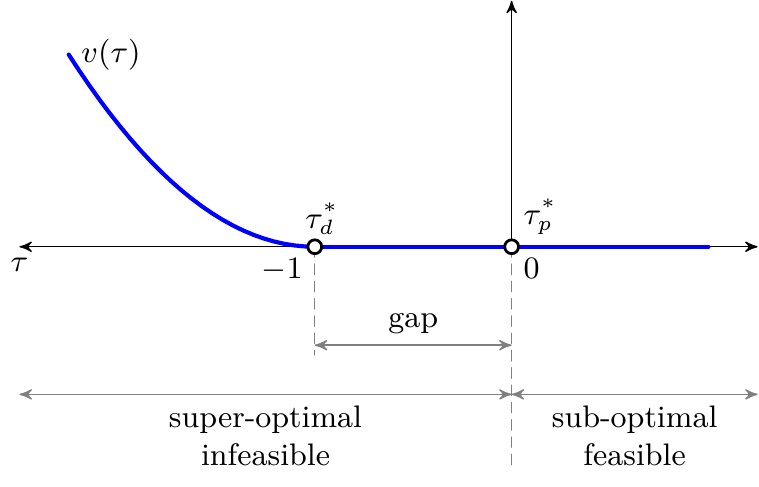}
  \caption{A depiction of a value function $v$ that exhibits the strict
    inequality described by~\eqref{eq:duality-gap}; see also
    \cref{sec:example_2}. In this example, the value function $v(\tau)$ vanishes
    for all $\tau\ge\dopt$, where $\dopt<\popt$.  Solutions of~\eqref{eq:Pf} for
    values of $\tau<\popt$ are necessarily super-optimal and infeasible
    for~\eqref{eq:cvx_primal}. The difference between $\dopt$ and $\popt$
    corresponds to the gap between the optimal values of \eqref{eq:cvx_primal} and its dual problem.}
  \label{fig:val-func-sdp}
\end{figure}

\subsection{Level-set methods}

In practice, only an approximate solution of the
problem~\eqref{eq:cvx_primal} is required, and the level-set method can
be used to obtain an approximate root that satisfies
$v(\tau)\le\epsilon$. The solution $x\in\Xscr$ of the corresponding
level-set problem~\eqref{eq:cvx_flipped} is super-optimal and
$\epsilon$-infeasible:
\begin{equation*}
  f(x) \le \popt \textt{and} g(x) \le \epsilon.
\end{equation*}
\citet{2016aravkinlevel} describe the general level-set approach, and establish
a complexity analysis that asserts that $\BigOh\big(\log\epsilon^{-1}\big)$
approximate evaluations of $v$ are required to obtain an $\epsilon$-infeasible
solution. These root-finding procedures are based on standard approaches,
including bisection, secant, and Newton methods. The efficiency of these
approaches hinges on the accuracy required of each evaluation of the value
function $v$. Aravkin et al. also demonstrate that the required complexity can be
achieved by requiring a bound on error in each evaluation of $v$ that is
proportional to $\epsilon$.

The formulation \eqref{eq:cvx_primal} is very general, even though the
constraint $g(x)\le0$ represents only a single function of the full constraint
set represented by $\Xscr$. There are various avenues for reformulating any
combination of constraints that lead to a single functional-constraint
formulation such as \eqref{eq:cvx_primal}. For instance, multiple linear
constraints of the form $Ax=b$ can be represented as a constraint on the norm of
the residual, i.e., $g(x) = \norm{Ax-b} \le 0$. More generally, for any set of
constraints $c(x)\le0$ where $c=(c_i)$ is a vector of convex functions $c_i$, we may
set $g(x) = \rho(\max\{0,\,c(x)\})$ for any convenient nonnegative convex
function $\rho$ that vanishes only at the origin, thus ensuring that $g(x)\le0$
if and only if $c(x)\le0$.

\section{Examples} \label{sec:examples}

We provide concrete examples that exhibit the behavior shown
in~\eqref{eq:duality-gap}. These semidefinite programs (SDPs) demonstrate that the
level-set method can produce diverging iterates.

Let $x_{ij}$ denote the $(i,j)th$
entry of the $n$-by-$n$ symmetric matrix $X=(x_{ij})$. The notation
$X\succeq0$ denotes the requirement that $X$ is symmetric positive
semidefinite.

\begin{example}[SDP with infinite gap] \label{sec:example_1}
Consider the $2\times2$ SDP
\begin{equation}
  \label{eq:sdp1}
  \minimize{X\succeq0} \enspace -2 x_{21}
  \enspace\st\enspace x_{11} = 0,
\end{equation}
whose solution and optimal value are given, respectively, by
\[
  X_* = \bmat{0 & 0 \\ 0 & 0} \quad\mbox{and}\quad \popt = 0.
\]
The Lagrange dual is a feasibility problem:
\[
  \maximize{y\in\Real} \enspace 0
  \enspace\st\enspace
  \bmat{\phantom-y & -1 \\ -1 & \phantom-0} \succeq 0.
\]
Because the dual problem is infeasible, we assign the dual optimal
value $\dopt = -\infty$. Thus, $\dopt = -\infty < \popt=0$, and this dual pairing fails to have strong duality.

The application of the level-set method to the primal problem~\eqref{eq:sdp1} can be accomplished by defining the functions
\[
  f(X) := -2 x_{21} \text{and} g(X) := |x_{11}| ,
\]
which together define the value function of the level-set problem~\eqref{eq:cvx_flipped}:
\begin{equation}
\label{eq:123}
  v(\tau) = \inf_{X\succeq0}
  \big\{\;\abs{x_{11}}
    \ \big|\
    {-2x_{21}}\le\tau\;
    \big\}.
\end{equation}
Because $X^*$ is primal optimal, $v(\tau) = 0$ for all
$\tau \geq \popt=0$. Now consider the parametric matrix
\[
  X(\tau, \epsilon) :=
  \bmat{\epsilon & \frac\tau{2_{\vphantom2}} \\[3pt] \frac\tau2 & \frac{\tau^2}{4\epsilon} }
  \textt{for all} \mbox{$\tau<0$ and $\epsilon > 0$},
\]
which is feasible for the level-set problem~\eqref{eq:123}. Thus,
$v(\tau)$ is finite. The level-set problem clearly has a zero lower
bound that can be approached by sending $\epsilon\downarrow0$. Thus,
$v(\tau) = 0$ for all $\tau < 0$.

In summary, $v(\tau) = 0$ for all
$\tau$, and so $v(\tau)$ has roots less than the true optimal
value $\popt$. Furthermore, for $\tau < 0$, there is no primal attainment for
\eqref{eq:Pf}, because $\lim_{\epsilon\downarrow0}X(\tau, \epsilon)$
does not exist.\qed
\end{example}

\begin{example}[SDP with finite gap] \label{sec:example_2}
Consider the $3\times3$ SDP
\begin{equation}
  \label{eq:sdp2}
  \minimize{X\succeq0} \enspace -2 x_{31}
  \enspace\st\enspace
  x_{11} = 0,\ x_{22} + 2 x_{31} = 1.
\end{equation}
The positive semidefinite constraint on $X$, together with the
constraint $x_{11}=0$, implies that $x_{31}$ must vanish. Thus, the
solution and optimal value are given, respectively, by
\begin{equation}\label{eq:2}
  X^* = \bmat{0 & 0 & 0 \\ 0 & 1 & 0 \\ 0 & 0 & 0}
  \textt{and}
  \popt = 0.
\end{equation}
The Lagrange dual problem is
\begin{equation*}
  \maximize{y\in\Real^2} \enspace -y_2
  \enspace\st\enspace
  \bmat{y_1 & 0 & y_2-1 \\ 0 & y_2 & 0 \\ y_2 - 1 & 0 & 0} \succeq 0.
\end{equation*}
The dual constraint requires $y_2 = 1$, and thus the optimal dual
value is $\dopt = -1 < 0 = \popt$.

For the application of the level-set method to
primal problem~\eqref{eq:sdp2}, we assign
\begin{equation}\label{eq:1}
  f(X) := -2x_{31}
  \enspace\mbox{and}\enspace
  g(X) :=  x_{11}^2 + (x_{22} + 2x_{31} - 1)^2,
\end{equation}
which together define the value function
\begin{equation}
  \label{eq:3}
  v(\tau) = \inf_{X\succeq0}\set{x_{11}^2 + (x_{22} + 2x_{31} - 1)^2 | -2x_{31} \le \tau}.
\end{equation}
As in \cref{sec:example_1}, any convex nonnegative $g$
function that vanishes on the feasible set could have been used to
define $v$. It follows from~\eqref{eq:2} that $v(\tau) = 0$ for all
$\tau \geq 0$. Also, it can be verified that $v(\tau) = 0$ for all
$\tau\ge\dopt=-1$. To understand this, first define the parametric matrix
\[
  X_\epsilon
  = \bmat{%
      \epsilon & 0 & \frac{1}{2}
    \\ 0 & 0 & 0
    \\ \frac{1}{2} & 0 & \frac{1}{4\epsilon}
  }
  \textt{with}
  \epsilon > 0,
\]
which is feasible for level-set problem~\eqref{eq:3}, and has
objective value $g(X_\epsilon) = \epsilon^2$. Because $X_\epsilon$ is
feasible for all positive $\epsilon$, the optimal value vanishes
because $v(\tau) = \inf\set{g(X_\epsilon)|\epsilon>0} = 0$. Moreover,
the set of minimizers for~\eqref{eq:3} is empty for all
$\tau\in(-1,0)$. \cref{fig:val-func-sdp} illustrates the behavior of
this value function.

Thus, we can produce a
sequence of matrices $X_\epsilon$ each of which is
$\epsilon$-infeasible with respect to the infeasibility measure
given by~\eqref{eq:1}. However, the limit as $\epsilon\downarrow0$
does not produce a feasible point, and the limit does not even exist
because the entry $x_{33}$ of $X_\epsilon$ goes to infinity.

The level-set method fails since the root of $v(\tau)$ identifies an
incorrect optimal primal value $\popt$, and instead identifies the
optimal dual value $\dopt<\popt$. \qed

\end{example}

\section{Value functions}

The level-set method based on~\eqref{eq:Pf} is founded on the inverse-function
relationship between the pair of ``flipped'' value functions
\begin{subequations} \label{eq:5}
\begin{align}
  p(u) &= \inf_{x\in\Xscr}\set{f(x) | g(x)\le u} \label{eq:5-p-func}
\\v(\tau)   &= \inf_{x\in\Xscr}\set{g(x) | f(x)\le\tau}.  \label{eq:5-v-func}
\end{align}
\end{subequations}
Clearly, $\popt=p(0)$. Here we summarize the key aspects of the  
relationship between the value functions $v$ and $p$, and their
respective solutions. \citet{AravkinBurkeFriedlander:2013} provide a complete description.

Let $\argmin\,v(\tau)$ and $\argmin\,p(u)$, respectively, denote the set of
solutions to the optimization problem underlying the value functions $v$ and
$p$. Thus, for example, if the value $p(u)$ is finite,
\[
  \argmin\,p(u) =
    \set{x\in\Xscr | f(x)=p(u),\ g(x)\le0};
\]
otherwise, $\argmin p(u)$ is empty.  Clearly,
$\argmin\,p(0)=\argmin\,\eqref{eq:cvx_primal}$.  Because $p$ is defined
via an infimum, $\argmin p(u)$ can be empty even if $p$ is
finite, in which case we say that the value $p(u)$ is not
attained.

Let $\Sscr$ be the set of parameters $\tau$ for which the level-set
constraint $f(x)\le\tau$ of~\eqref{eq:cvx_flipped} holds with
equality. Formally,
  \begin{equation*}
  \Sscr
  = \big\{ \tau \le+\infty \mid \emptyset \neq \argmin v(\tau)
    \subseteq \set{ x \in \Xscr | f(x) = \tau} \big\}.
\end{equation*}
The following theorem establishes the relationships between the value functions
$p$ and $v$, and their respective solution sets. This result is reproduced from \citet[Theorem 2.1]{AravkinBurkeFriedlander:2013}.

\begin{theorem}[Value-function inverses]
  \label{th:value-inverse}
  For every $\tau\in\Sscr$, the following statements hold:
\begin{enumerate}
  \item[\rm(a)] $(p\circ v)(\tau)=\tau$,
  \item[\rm(b)] $\argmin v(\tau)
             = \argmin\, (p\circ v)(\tau)
               \subseteq \set{x \in \Xscr | f(x) = v(\tau) }$.
\end{enumerate}
\end{theorem}
The condition $\tau\in\Sscr$ means that the constraint of the
level-set problem~\eqref{eq:cvx_flipped} must be active in order for
the result to hold. The following example establishes that this condition is necessary.

\begin{example}[Failure of value-function inverse]
  The univariate problem
\[
  \minimize{x\in\Real}\enspace\abs x \enspace\st\enspace \abs{x}-1\le0
\]
has the trivial solution $x^*=0$ with optimal value $\popt=0$. Note that the
constraint is inactive at the solution, which violates the hypothesis of
\cref{th:value-inverse}.
Now consider the value
functions
\begin{align*}
  p(u) &= \inf\,\{\ \abs x \,:\, \abs x-1\le u\ \},
\\v(\tau)   &= \inf\,\{\ \abs x - 1 \,:\, \abs x\le\tau   \},
\end{align*}
which correspond, respectively, to a parameterization of the original problem,
and to the level-set problem. The level-set value function $v$ evaluates to
\[
  v(\tau) =
  \begin{cases}
    -1      & \mbox{if $\tau\ge\popt$}
  \\+\infty & \mbox{if $\tau<\popt$.}
  \end{cases}
\]
Because $p$ is nonnegative over its domain, there is no value $\tau$ for which the
inverse-function relationship shown by \cref{th:value-inverse}(a) holds.
\end{example}

\cref{th:value-inverse} is symmetric, and holds if the roles of $f$ and $g$, and
$p$ and $v$, are reversed. \citet{AravkinBurkeFriedlander:2013} show that
this result holds even if the underlying functions and sets that
define~\eqref{eq:cvx_primal} are not convex.

Part (b) of the theorem confirms that if $\popt\in\Sscr$, i.e., the
constraint $g(x)\le0$ holds with equality at a solution of~\eqref{eq:cvx_primal},
then solutions of the level-set problem
coincide with solution of the original problem defined by $p(0)$. More
formally,
\[
  \argmin v(\popt)=\argmin\,\eqref{eq:cvx_primal}.
\]

Again consider \cref{sec:example_2}, where we set $\tau = -1/2$, which
falls midway between the interval $(\dopt,\popt)=(-1,0)$. Because the
solution set $\argmin v(\tau)$ is empty, $\tau\notin\Sscr$. Thus,
\[
  (p\circ v)(\tau) = p(0) = 0 \neq \tau,
\]
and the level-set method fails.

In order establish an inverse-function-like relationship between the
value functions $p$ and $v$ that always holds for convex problems, we
provide a modified definition of the epigraphs for $v$ and $w$.

\begin{definition}[Value function epigraph]
  The {\em value function epigraph} of the optimal value function
  $p$ in~\eqref{eq:5-p-func} is defined by
  \[
    \vfepi p =
    \Set{%
      (u,\tau) | \exists x\in\Xscr, \, f(x)\le\tau,\ g(x) \leq u
    }.
  \]
\end{definition}

This definition similar to the regular definition
for the epigraph of a function, given by
\[
  \epi p = \Set{ (u,\tau) | p(u)\le\tau },
\]
except that if $\tau = p(u)$ but $\argmin p(u)$ is empty,
then $(u, \tau) \notin \vfepi w$.

The result below follows immediately from the definition of the value
function epigraph.  It establishes that \eqref{eq:11} holds if~\eqref{eq:cvx_flipped} has a solution that attains its optimal value
(as opposed to relying on the infimal operator to achieve that value).

\begin{proposition} For the value functions $p$ and $v$,
\[
  (u, \tau) \in \vfepi p \iff (\tau, u) \in \vfepi v.
\]
\end{proposition}

\section{Duality in convex optimization} \label{sec:duality-framework}

Duality in convex optimization can be understood as describing the
behavior of an optimization problem under perturbation to its
data. From this point of view, dual variables describe the sensitivity
of the problem's optimal value to that perturbation. The description
that we give here summarizes a well-developed theory fully described
by~\citet{rtrw:1998}. We adopt a geometric
viewpoint that we have found helpful for understanding the connection
between duality and the level-set method, and lays out the objects needed for
the analysis in subsequent sections.

For this section only, consider the generic convex optimization problem
\[
  \minimize{x\in\Xscr}\enspace h(x),
\]
where $h:\Real^n\to\Real\cup\{\infty\}$ is an arbitrary closed proper convex function.
The perturbation approach is
predicated on fixing a certain convex function
$F(x,u):\Real^n\times\Real^m\to\Real\cup\{\infty\}$ with the property
that
\[
  F(x,0) = h(x) \quad \forall x.
\]
Thus, the particular choice of $F$ determines the perturbation function
\[
  p(u) := \inf_{x} F(x,u),
\]
which describes how the optimal value of $h$ changes under a
perturbation $u$. We seek the behavior of the perturbation function
about the origin, at which the value of $p$ coincides with the optimal
value $\popt$, i.e., $p(0)=\popt$.

The convex conjugate of the function $p$ is
\[
  \conj p(\mu) = \sup_u\set{\ip \mu u - p(u) }
\]
defines the affine function $\mu\mapsto\ip \mu u - \conj p(\mu)$ that minorizes
$p$ and supports the epigraph of $p$; see \cref{fig:dual-value-functions}. The
biconjugate $\biconj p$
provides a convex and closed function that is a global lower envelope for $p$,
i.e., $\biconj p(u)\le p(u)$ for all $u$. This last inequality is tight at a point $u$, i.e.,
$\biconj p(u)=p(u)$, if and only if $p$ is lower-semicontinuous at
$u$ \cite[Theorem~7.1]{roc70}. Because of the connection between lower semicontinuity and the
closure of the epigraph, we say that $p$ is \emph{closed} at such
points $u$.

As described by \citet[Lemma 11.38]{rtrw:1998}, the function $p$ and its
biconjugate $\biconj p$ define dual pairs of optimization problems given by
\begin{equation} \label{eq:12}
  p(0) = \inf_x\, F(x,0)
  \quad\mbox{and}\quad
  \biconj p(0) = \sup_y\, {-\conj F(0,y)},
\end{equation}
which define the primal and dual optimal values
\begin{equation} \label{eq:6}
  \dopt \coloneqq \biconj p(0) \le p(0) =: \popt.
\end{equation}
Strong duality holds when
$\popt=\dopt$, which indicates the closure of $p$ at the origin.
As we show in \cref{sec:value-duality}, the optimal dual value $\dopt$ coincides
with the value of the infimal value defined in~\eqref{eq:duality-gap}.

\begin{figure*}[t]
  \centering
  \begin{tabular}{@{}c@{\hspace{.3in}}c@{}}
     \includegraphics[page=4]{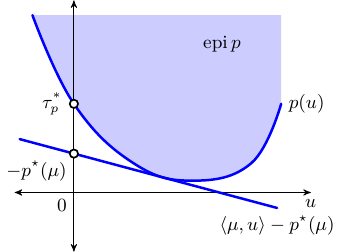} 
    &\includegraphics[page=3]{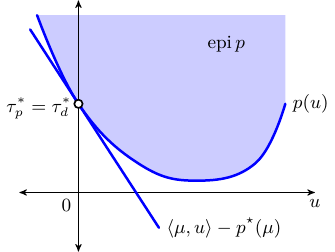}
\\ (a) Non-optimal dual & (b) Optimal dual
  \end{tabular}
  \caption{The relationship between the primal perturbation value $p(u)$ and
    a single instance (with slope $\mu$ and intercept $q_\mu$) of the
    uncountably many minorizing affine functions that define the dual
    problem. The panel on the left depicts a non-optimal supporting
    hyperplane that crosses the vertical axis at $-\conj p(\mu)<\popt$; the panel on the
    right depicts an optimal supporting hyperplane that generates a
    slope $\mu$ and intercept $-\conj p(\mu)=\popt$.
    \label{fig:dual-value-functions}}
\end{figure*}

The following well-known result establishes a constraint qualification for
\eqref{eq:cvx_primal} that ensures strong duality holds. See \citet[Theorem
11.39]{rtrw:1998} for a more comprehensive version of this result. 
\begin{theorem}[Weak and strong duality] \label{th:duality}
  Consider the primal-dual pair~\eqref{eq:12}.
  \begin{enumerate}
  \item[\rm(a)] {\normalfont[Weak duality]} The inequality $\popt \ge \dopt$ always holds.
  \item[\rm(b)] {\normalfont[Strong duality]} If $0\in\interior\dom p$,
    then $\popt=\dopt$.
  \end{enumerate}
\end{theorem}

To establish the connection between the pair of value functions~\eqref{eq:5}
for~\eqref{eq:cvx_primal} and this duality framework, we observe that
\[
  p(u) = \inf_{x\in\Xscr}\,\set{f(x)|g(x)\le u}
  = \inf_x\,F(x,u),
\]
where 
\begin{equation} \label{eq:16}
  F(x,u) \coloneqq f(x) + \delta_\Xscr(x) + \delta_{\epi g}(x, u),
\end{equation}
and the indicator function $\delta_{\Cscr}$ vanishes on the set $\Cscr$ and is
$+\infty$ otherwise. The dual problem $\biconj p(0)$ defined in~\eqref{eq:12} is
derived as follows:
\begin{equation} \label{eq:15}
  \begin{aligned}
    \biconj p(0) &= \sup_\lambda\, -\conj F(0,\lambda) 
    \\ &= \sup_\lambda\, \inf_{x,u}\set{ f(x) + \delta_{\Xscr}(x) - \lambda u + \delta_{\epi g}(x,u) }
    \\ &= \sup_{\lambda\le0} \inf_{x\in\Xscr}\set{ f(x) - \lambda g(x)}.
  \end{aligned}
\end{equation}
We recognize this last expression as the familiar Lagrangian-dual for the
optimization problem~\eqref{eq:cvx_primal}.

\section{Duality of the value function root} \label{sec:value-duality}

We now provide a formal statement and proof our main result concerning
problem~\eqref{eq:cvx_primal} and the inequality shown in~\eqref{eq:duality-gap}. In the
latter part of this section we also provide a straight-forward extension of the main
result that allows for multiple constraints, and not just a single constraint
function, as specified by~\eqref{eq:cvx_primal}.

Note that the theorem below does not address conditions under which $v(\popt)\le0$,
which is true if and only if the solution set $\argmin\,\eqref{eq:cvx_primal}$
is not empty. In particular, any $x^*\in\argmin\,\eqref{eq:cvx_primal}$ is a
solution of~\eqref{eq:cvx_flipped} for $\tau=\popt$, and hence $v(\popt)\le0$.
However, if $\argmin\,\eqref{eq:cvx_primal}$ is empty, then there is no solution
to~\eqref{eq:cvx_flipped} and hence $v(\popt)=+\infty$.

\begin{theorem}[Duality of the value function root]
  \label{thm:val-fnc-roots}
  For problem~\eqref{eq:cvx_primal} and the pair of value function $v$ and $p$, defined by~\eqref{eq:5},
  \[
    \dopt = \inf\set{\tau|v(\tau)\le0}
    \quad\mbox{and}\quad
    v(\tau) \le 0 \enspace\mbox{for all}\enspace \tau > \dopt,
  \]
  where $\dopt\coloneqq\biconj p(0)$ is the optimal value of the Lagrange-dual problem~\eqref{eq:15}.
\end{theorem}

Before giving the proof, below, we provide an intuitive argument for
\cref{thm:val-fnc-roots}. Suppose that strong duality holds for
\eqref{eq:cvx_primal}. Hence, $\popt=p(0) = p^{**}(0)=\dopt$, which means that
the perturbation function $p$ is closed at the origin. We sketch in the top row
of \cref{fig:nostrongduality} example pairs of value functions $p$ and $v$ that
exhibit this behavior. To understand this picture, first consider the value
$\tau_1 < \popt$, shown in the top row. It is evident that $v(\tau_1)$ is
positive, because otherwise there must exist a vector $x\in\Xscr$ that is
super-optimal and feasible, i.e.,
\[
  f(x)\le\tau_1<\popt \quad\mbox{and}\quad g(x)\le0,
\]
which
contradicts the definition of $\popt$. It then follows that the value
$u:=v(\tau_1)$ yields $p(u) = \tau_1$. For $\tau_2 > \ts$, any solution to the
original problem would be feasible (therefore requiring no perturbation $u$) and
would achieve objective value $p(0) = \popt < \tau_2$. Furthermore, notice that
as $\tau_1 \rightarrow \popt$, the value $p(u_1)$ varies continuously in
$\tau_1$, where $u_1$ is the smallest root of $p(u) = \tau_1$.

\begin{figure*}[t]
  \centering
  \begin{tabular}{@{}c@{\hspace{0.5in}}c@{}}
     \includegraphics[page=2]{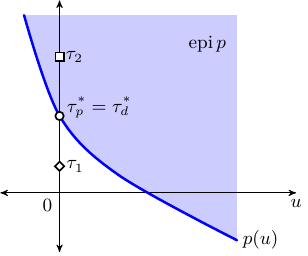}
    &\includegraphics[page=6]{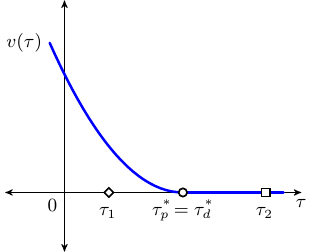}
    \\\parbox[t]{2in}{(a) Perturbation function $p$ under strong duality}
    & \parbox[t]{2in}{(b) Level-set value function $v$ corresponding to (a)}
      \vspace*{7pt}
    \\
     \includegraphics[page=5]{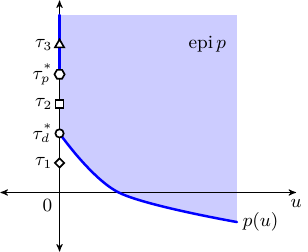}
    &\includegraphics[page=7]{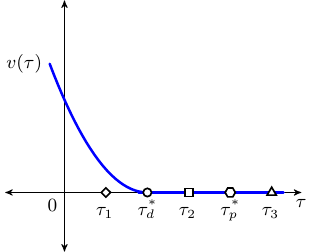}
    \\\parbox[t]{2in}{(c) Perturbation function $p$ with no strong duality}
    & \parbox[t]{2in}{(d) Level-set value function $v$ corresponding to (c)}
  \end{tabular}
  \caption{The perturbation function $p(u)$ and corresponding
    level-set value function $v(\tau)$ for problems with strong
    duality (top row) and no strong duality (bottom row). Panel (c)
    illustrates the case when strong duality fails and the graph of
    $p$ is open at the origin, which implies that
    $\dopt<\popt\equiv p(0)$.}
  \label{fig:nostrongduality}
\end{figure*}

Next consider the second row of Figure~\ref{fig:nostrongduality}. In
this case, strong duality fails, which means that
$$\lim_{u \downarrow 0} p(u) = \dopt \neq p(0).$$ With
$\tau = \tau_1$, we have $v(\tau_1) > 0$. With
$\tau = \tau_3 > \popt$, we have $v(\tau) = 0$ because any solution
to~\eqref{eq:cvx_primal} causes~\eqref{eq:cvx_flipped} to have zero
value. But for $\dopt < \tau_2 < \popt$, we see that $v(\tau_2) = 0$,
because for any positive $\epsilon$ there exists positive
$u < \epsilon$ such that $p(u) \leq \tau_2$. Even though there is no
feasible point that achieves a superoptimal value
$f(x) \leq \tau_2 < \popt$, for any positive~$\epsilon$ there exists
an $\epsilon$-infeasible point that achieves that objective
value.

\begin{proof}[Theorem~\ref{thm:val-fnc-roots}]

We first prove the second result that $v(\tau)\le0$ if
$\tau>\dopt$.  Suppose that strong duality holds, i.e.,
$\popt=\dopt$. Then the required result is immediate because if $\popt$ is the optimal value, then for any $\tau > \popt$, there exists feasible $x$ such that $f(x)\le\tau$.

Suppose that strong duality does not hold, i.e., $\popt > \dopt$. If
$\tau>\popt$, it is immediate that $v(\tau)\le0$. Assume, then, that
$\tau\in(\dopt,\popt]$. Note that the two conditions $g(x) \le u$ and $f(x) \le
\tau$ are equivalent to the single condition $F(x,u)\le\tau$, where $F$ is defined
by~\eqref{eq:16}. We will therefore prove that
\begin{equation} \label{eq:8}
    \forall \epsilon > 0,\ 
    \mbox{$\exists x\in\Xscr$ such that}\
    F(x,u) \leq \tau,\ u\le\epsilon,
\end{equation}
which is equivalent to the required condition $v(\tau) \leq 0$.  It
follows from the convexity of $\epi p$ and from \eqref{eq:6} that
$(0,\dopt)\in\epi p^{**} = \cl \epi p$. Thus,
\[
    \forall \eta > 0,\
    \mbox{$\exists (u,\omega) \in \epi p$ such that}\
    \|(u,\omega) - (0, \dopt)\| < \eta.
\]
Note that
\begin{equation}  \label{eq:13}
\begin{aligned}
    \lim_{\epsilon\downarrow0}\inf\left\{p(u) \,\big|\, |u|\le \epsilon\right\}
    &\overset{\rm(i)}=
    \lim_{\epsilon\downarrow0}\inf\left\{\biconj p(u)\,\big|\,|u|\le \epsilon\right\}
    \\&\overset{\rm(ii)}=
    \biconj p(0)
    \overset{\rm(iii)}=\
    \dopt,
  \end{aligned}
\end{equation}
where equality (i) follows from the fact that $p(u) = \biconj p(u)$
for all $u\in\dom p$, equality (ii) follows from the closure of
$\biconj p$, and (iii) follows from \eqref{eq:6}.
This implies that
\[
    \forall \eta > 0,\
    \mbox{$\exists (u,\omega) \in \epi p$ such that}\
    \|(u,p(u)) - (0, \dopt)\| < \eta.
\]
For any fixed positive $\epsilon$ define
$\mu = \min \set{ \epsilon,\ \frac{1}{4} (\tau - \dopt)}$. Choose
$\hat u\in\dom p$ such that
$\|(\hat u,p(\hat u)) - (0, \dopt)\| < \mu$, and so
\[
  \epsilon\ge\mu > \|(\hat u,p(\hat u)) - (0, \dopt)\|
  \geq
  \max\left\{\,\norm{\hat u} ,\, |p(\hat u)-\dopt|\,\right\}.
\]
Thus,
\begin{equation} \label{eq:7}
      p(\hat u) < \dopt + \mu.
 \end{equation}
 Moreover, it follows from the definition of
 $p(\hat u)$, cf.~\eqref{eq:5-p-func}, that
 \begin{align*}
      \forall \nu > 0,\
      \mbox{$\exists x\in\Xscr$
      such that $F(x,\hat u) \leq p(\hat u) + \nu$}.
\end{align*}
Choose $\nu = \mu$, and so there exists $\hat x$ such that
$F(\hat x, \hat u) \leq p(\hat u) + \mu$. Together with \eqref{eq:7},
we have
\[
      f(\hat x) \leq p(\hat u) + \mu < \dopt + 2\mu \leq \tau.
\]
Therefore, for each $\epsilon > 0$, we can find a pair
$(\hat x,\hat u)$ that satisfies \eqref{eq:8}, which completes
the proof of the second result.

Next we prove the first result, which is equivalent to proving
that $v(\tau)>0$ if $\tau<\dopt$ because $v(\tau)$ is convex. Observe that
$\tau < \dopt \equiv p^{**}(0)$ is equivalent to
$(0,\tau) \notin \cl \epi p$, which implies that
\begin{equation} \label{eq:14}
\begin{aligned} 
  0 &< \inf_u \set{u | (u,\tau) \in \cl \epi p}
\\&= \inf_u \set{u | (u,\tau) \in \epi p}
  \\&= \inf_u \set{u | \exists x \in \Xscr \enspace\mbox{such that}\enspace F(x,u) \le \tau} = v(\tau),
\end{aligned}
\end{equation}
which completes the proof. \qed
\end{proof}
The proof of \cref{thm:val-fnc-roots} reveals that the behavior
exhibited by \cref{sec:example_1,sec:example_2} stems from the failure of
strong duality with respect to perturbations in the linear
constraints.

\subsection{General perturbation framework} \label{sec:general-duality}

We now generalize \cref{thm:val-fnc-roots} to inlclude arbitrary perturbations
to \eqref{eq:cvx_primal}, and thus more general notions of duality. In this case
we are interested in the value function pair
\begin{subequations} \label{eq:56}
\begin{align}
  p(u) &= \inf_{x\in\Xscr} F(x,u),
\\v(\tau) &= \inf_{x\in\Xscr}\set{\norm{u}|F(x,u)\le\tau},
\end{align}
\end{subequations}
where $F:\Real^n\times\Real^m\to\Real\cup\set{\infty}$ is an arbitrary convex
function with the property that $F(x,0)=f(x)$
(cf.~\cref{sec:duality-framework}), and $\norm{\cdot}$ is any norm. Because $p$
is parameterized by an $m$-vector $u$ and not just a scalar as previously
considered, we must consider the norm of the perturbation. Therefore, $v(\tau)$
is necessarily non-negative. We are thus interested in the leftmost root of the
equation $v(\tau) = 0$, rather than an inequality as in
\cref{thm:val-fnc-roots}.

\begin{example}[Multiple constraints] \label{ex:lvl-pert}
  Consider the convex optimization problem
  \begin{equation} \label{eq:lagrange-example}
    \minimize{x} \enspace f(x) \enspace\st\enspace c(x) \leq 0,\ Ax=b,
  \end{equation}
  where $c=(c_i)_{i=1}^m$ is a vector-valued convex function and $A$ is a matrix.
  Introduce perturbations $u_1$ and $u_2$ to the right-hand sides of
  the constraints, which gives rise to Lagrange duality, and
  corresponds to the perturbation function
  \[
    p(u_1, u_2) = \inf_x \set{ f(x) | c(x) \leq u_1,\enspace Ax - b = u_2 }.
  \]
  One valid choice for the value function that corresponds to swapping
  both constraints with the objective to~\eqref{eq:lagrange-example}
  can be expressed as
  \[
    v(\tau) = \inf_{x,u_1,u_2}
    \left\{
      \tfrac{1}{2}\|[u_1]_+\|_2^2 + \tfrac{1}{2} \|u_2\|_2^2
      \ \middle|\
      \begin{aligned}
        f(x) &\leq \tau\\ c(x)&\le u_1\\ Ax -b &=  u_2
        \end{aligned}
    \right\},
  \]
  where the operator $[u_1]_+=\max\{0, u_1\}$ is taken component-wise
  on the elements of $u_1$. This particular formulation of the value
  function makes explicit the connection to the perturbation
  function. We may thus interpret the value function as giving the
  minimal perturbation that corresponds to an objective value
  less than or equal to~$\tau$. \qed
\end{example}

\begin{theorem}
\label{thm:val-fnc-roots-general}
For the functions $p$ and $v$ defined by~\eqref{eq:56},
\[
  \dopt = \inf\set{\tau|v(\tau)=0}
  \quad\mbox{and}\quad
  v(\tau) = 0 \enspace\mbox{for all}\enspace \tau > \dopt.
\]
\end{theorem}
The proof is almost identical to that of \cref{thm:val-fnc-roots},
except that we treat $u$ as a vector, and replace $u$ by $\norm{u}$ in
\eqref{eq:8}, \eqref{eq:13}, and~\eqref{eq:14}.

\cref{thm:val-fnc-roots,thm:val-fnc-roots-general} imply that $v(\tau) \le 0$ for all values larger than the
optimal dual value. (The inequality $\tau > \dopt$ is strict, as
$v(\dopt)$ may be infinite.) Thus if strong duality does not hold,
then $v(\tau)$ identifies the wrong optimal value for the original
problem being solved. This means that the level-set method may
provide a point arbitrarily close to feasibility, but is at
least a fixed distance away from the true solution independent of how
close to feasibility the returned point may be. %

\begin{example}[Basis pursuit denoising \citep{cds98,chendonosaun:2001}]
  The level-set method implemented in the SPGL1 software package
  solves the 1-norm regularized least-squares problem
  \[
    \minimize{x} \enspace \norm{x}_1
    \enspace\st\enspace
    \norm{Ax-b}_2 \le u
  \]
  for any value of $u\ge0$, assuming that the problem remains
  feasible. (The case $u=0$ is important, as it accommodates the
  case in which we seek a sparse solution to the under-determined linear
  system $Ax=b$.) The algorithm approximately solves a sequence of
  flipped problems
  \[
    \minimize{x} \enspace \norm{Ax-b}_2
    \enspace\st\enspace
    \norm{x}_1 \le \tau_k,
  \]
  where $\tau_k$ is chosen so that the corresponding solution $x_k$
  satisfies $\norm{Ax_k-b}_2\approx u$.  Strong duality holds
  because the domains of the nonlinear functions (i.e., the 1- and
  2-norms) cover the whole space. Thus, the level-set method succeeds on this
  problem. \qed
\end{example}

\section{Sufficient conditions for strong duality}

The condition that $0\in\dom p$ may be interpreted as Slater's
constraint qualification \citep[\S3.2]{borwein2010convex}, which in
the context of~\eqref{eq:cvx_primal} requires that there exist a point
$\hat x$ in the domain of $f$ and for which $g(\hat x)<0$. This
condition is sufficient to establish strong duality. Here we show how
\cref{thm:val-fnc-roots} can be used as a device to characterize an
alternative set of sufficient conditions that continue to ensure
strong duality even for problems that do not satisfy Slater's
condition.

\begin{proposition}
\label{thm:strong-duality-holds}
Problem~\eqref{eq:cvx_primal} satisfies strong duality if either one of
the following conditions hold:
\begin{enumerate}
\item [\rm(a)] the objective $f$ is coercive, i.e., $f(x) \rightarrow \infty$ as $\|x\| \rightarrow \infty$;
\item [\rm(b)] $\Xscr$ is compact.
\end{enumerate}
\end{proposition}
\begin{proof}
  Consider the level-set problem~\eqref{eq:cvx_flipped} and its
  corresponding optimal-value function $v(\tau)$ given
  by~\eqref{eq:Pf}.  In either case (a) or (b), the feasible set
  \[\set{x\in\Xscr|f(x)\le\tau}\] of \eqref{eq:Pf} is compact because
  either $\Xscr$ is compact or the level sets of $f$ are
  compact. Therefore,~\eqref{eq:cvx_flipped} always attains its
  minimum for all $\tau \geq \inf\set{f(x) | x\in\Xscr}$.

  Suppose strong duality does not hold. \cref{thm:val-fnc-roots} then
  confirms that there exists a parameter $\tau\in(\dopt,\popt)$ such
  that $v(\tau) = 0$. However, because~\eqref{eq:cvx_flipped} always
  attains its minimum, there must exist a point $\hat x \in \Xscr$
  such that $f(\hat x) \leq \tau < \popt$ and $g(x) \leq 0$, which
  contradicts the fact that $\popt$ is the optimal value
  of~\eqref{eq:cvx_primal}. We have therefore established that
  $\dopt = \popt$ and hence that~\eqref{eq:cvx_primal} satisfies
  strong duality. \qed
\end{proof}

We can use \cref{thm:strong-duality-holds} to establish that certain
optimization problems that do not satisfy a Slater constraint
qualification still enjoy strong duality. As an example, consider the conic optimization problem
\begin{equation} \label{eq:9} 
  \minimize{x} \enspace \ip c x \enspace\st\enspace \Ascr x=b,\ x\in\Kscr,
\end{equation}
where $\Ascr:\Escr_1\to E_2$ is a linear map between Euclidean spaces
$\Escr_1$ and $\Escr_2$, and $\Kscr\subseteq\Escr_1$ is a closed
proper convex cone. This wide class of problems includes linear
programming (LP), second-order programming (SOCP), and SDPs, and has
many important scientific and engineering applications
\citep{ben-nemi:2001}. If $c$ is in the interior of the dual
cone $\Kscr^*=\set{y\in\Escr_1|\ip x y \ge 0\ \forall x\in\Kscr}$,
then $\ip c x >0$ for all feasible $x\in\Kscr$. Equivalently, the
function $f(x):=\ip c x + \delta_\Kscr(x)$ is
coercive. Thus,~\eqref{eq:9} is equivalent to the problem
\[
  \minimize{x} \enspace f(x) \enspace\st\enspace \Ascr x=b,
\]
which has a coercive objective. Thus, Part (a) of
\cref{thm:strong-duality-holds} applies, and strong duality holds.

A concrete application of this model problem is the SDP
relaxation of the celebrated phase-retrieval problem~\citep{csv2013,waldspurger2015phase}
\begin{equation}
  \label{eq:phaselift}
  \minimize{X} \enspace \trace(X) \enspace\st\enspace \Ascr X=b,\ X\succeq0,
\end{equation}
where $\Kscr$ is now the cone of Hermitian
positive semidefinite matrices (i.e., all the eigenvalues are
real-valued and nonnegative) and $c=I$ is the identity matrix, so that
$\ip C X = \trace(X)$. In that setting, \citet{csv2013} prove that
with high probability, the feasible set of~\eqref{eq:9} is a rank-1
singleton (the desired solution), and thus we cannot use Slater's
condition to establish strong duality. However, because $\Kscr$ is
self dual \citep[Example 2.24]{bv:2004}, clearly $c\in\interior\Kscr$,
and by the discussion above, we can use
\cref{thm:strong-duality-holds} to establish that strong duality
holds~\eqref{eq:phaselift}.

A consequence of \cref{thm:strong-duality-holds} is that it is
possible to modify~\eqref{eq:cvx_primal} in order to guarantee strong
duality. In particular, we may regularize the objective, and instead
consider a version of the problem with the objective as
$f(x) + \mu \|x\|$, where the parameter $\mu$ controls the degree of
regularization contributed by the regularization term $\|x\|$. If, for
example, $f$ is bounded below on $\Xscr$, the regularized objective is
then coercive and \cref{thm:strong-duality-holds} asserts that the
revised problem satisfies strong duality. Thus, the optimal value
function of the level-set problem has the correct root, and the
level-set method is applicable. For toy problems such as
\cref{sec:example_1,sec:example_2}, where all of the feasible
points are optimal, regularization would not perturb the solution;
however, in general we expect that the regularization will perturb the
resulting solution, and in some cases this may be the desired
outcome.

\begin{acknowledgements}
  The authors are indebted to Professor Bruno F. Louren\c{c}o of Seikei
  University for fruitful discussions that followed the second author's
  course on first-order methods at the summer school associated with the
  2016 International Conference on Continuous Optimization, held in
  Tokyo. Professor Louren\c{c}o asked if level-set methods could be be
  applied to solve degenerate SDPs. His thinking was that the level-set
  problems~\eqref{eq:cvx_flipped} might satisfy Slater's constraint
  qualification even if the original problem~\eqref{eq:cvx_primal} did
  not, and therefore the level-set method might be useful as a way to
  alleviate numerical difficulties that can arise when an algorithm is
  applied directly to an SDP without strong duality. The conclusion of
  this paper suggests that this is not always the case. We also give sincere
  thanks to two anonymous referees for their many helpful suggestions, and to
  the Associate Editor, Tibor Csendes.
\end{acknowledgements}

\bibliography{shorttitles,master,friedlander,bib}

\providecommand{\noopsort}[1]{}
\begin{thebibliography}{23}
\providecommand{\natexlab}[1]{#1}
\providecommand{\url}[1]{{#1}}
\providecommand{\urlprefix}{URL }
\expandafter\ifx\csname urlstyle\endcsname\relax
  \providecommand{\doi}[1]{DOI~\discretionary{}{}{}#1}\else
  \providecommand{\doi}{DOI~\discretionary{}{}{}\begingroup
  \urlstyle{rm}\Url}\fi
\providecommand{\eprint}[2][]{\url{#2}}

\bibitem[{Aravkin et~al.(2013)Aravkin, Burke, and
  Friedlander}]{AravkinBurkeFriedlander:2013}
Aravkin AY, Burke J, Friedlander MP (2013) Variational properties of value
  functions. {SIAM} J Optim 23(3):1689--1717, \doi{10.1137/120899157},
  \eprint{http://epubs.siam.org/doi/pdf/10.1137/120899157}

\bibitem[{{Aravkin} et~al.(2018){Aravkin}, {Burke}, {Drusvyatskiy},
  {Friedlander}, and {Roy}}]{2016aravkinlevel}
{Aravkin} AY, {Burke} JV, {Drusvyatskiy} D, {Friedlander} MP, {Roy} S (2018)
  {Level-set methods for convex optimization}. Math Program, Ser B
  174(1-2):359--390, \doi{10.1007/s10107-018-1351-8}

\bibitem[{Ben-Tal and Nemirovski(2001)}]{ben-nemi:2001}
Ben-Tal A, Nemirovski A (2001) Lectures on Modern Convex Optimization:
  {A}nalysis, Algorithms, and Engineering Applications, MPS/SIAM Series on
  Optimization, vol~2. Society of Industrial and Applied Mathematics,
  Philadelphia

\bibitem[{van~den Berg and Friedlander(2007)}]{spgl12007}
van~den Berg E, Friedlander MP (2007) {SPGL1}: A solver for large-scale sparse
  reconstruction. Http://www.cs.ubc.ca/labs/scl/spgl1

\bibitem[{van~den Berg and Friedlander(2008{\natexlab{a}})}]{bf2008}
van~den Berg E, Friedlander MP (2008{\natexlab{a}}) Probing the pareto frontier
  for basis pursuit solutions. SIAM Journal on Scientific Computing
  31(2):890--912, \doi{10.1137/080714488},
  \urlprefix\url{http://link.aip.org/link/?SCE/31/890}

\bibitem[{van~den Berg and
  Friedlander(2008{\natexlab{b}})}]{BergFriedlander:2008}
van~den Berg E, Friedlander MP (2008{\natexlab{b}}) Probing the {Pareto}
  frontier for basis pursuit solutions. {SIAM} J Sci Comput 31(2):890--912,
  \doi{10.1137/080714488}

\bibitem[{van~den Berg and Friedlander(2011)}]{BergFriedlander:2011}
van~den Berg E, Friedlander MP (2011) Sparse optimization with least-squares
  constraints. {SIAM} J Optim 21(4):1201--1229, \doi{10.1137/100785028}

\bibitem[{van~den Berg and Friedlander(2013)}]{BergFrie:2007b}
van~den Berg E, Friedlander MP (2013) {SPGL1}: A solver for large-scale sparse
  reconstruction. \urlprefix\url{https://www.cs.ubc.ca/~mpf/spgl1/}

\bibitem[{Borwein and Lewis(2010)}]{borwein2010convex}
Borwein J, Lewis AS (2010) Convex analysis and nonlinear optimization: theory
  and examples. Springer Science \& Business Media

\bibitem[{Boyd and Vandenberghe(2004)}]{bv:2004}
Boyd S, Vandenberghe L (2004) Convex optimization. Cambridge University Press,
  Cambridge, \doi{10.1017/CBO9780511804441}

\bibitem[{Cand\`es et~al.(2013)Cand\`es, Strohmer, and Voroninski}]{csv2013}
Cand\`es EJ, Strohmer T, Voroninski V (2013) Phase{L}ift: exact and stable
  signal recovery from magnitude measurements via convex programming. Comm Pure
  Appl Math 66(8):1241--1274, \doi{10.1002/cpa.21432},
  \urlprefix\url{http://dx.doi.org/10.1002/cpa.21432}

\bibitem[{Chen et~al.(1998)Chen, Donoho, and Saunders}]{cds98}
Chen SS, Donoho DL, Saunders MA (1998) Atomic decomposition by basis pursuit.
  {SIAM} J Sci Comput 20(1):33--61

\bibitem[{Chen et~al.(2001)Chen, Donoho, and Saunders}]{chendonosaun:2001}
Chen SS, Donoho DL, Saunders MA (2001) Atomic decomposition by basis pursuit.
  {SIAM} Rev 43(1):129--159

\bibitem[{Conn et~al.(2000)Conn, Gould, and Toint}]{conngoultoin:2000}
Conn AR, Gould NIM, Toint {\mbox{Ph}}L (2000) Trust-Region Methods. {MPS-SIAM}
  Series on Optimization, Society of Industrial and Applied Mathematics,
  Philadelphia

\bibitem[{Higham(2002)}]{higham2002accuracy}
Higham NJ (2002) Accuracy and stability of numerical algorithms, vol~80.
  Society for Industrial and Applied Mathematics

\bibitem[{Lemar\'echal et~al.(1995)Lemar\'echal, Nemirovskii, and
  Nesterov}]{lemaneminest:1995}
Lemar\'echal C, Nemirovskii A, Nesterov Y (1995) New variants of bundle
  methods. Math Program 69:111--147

\bibitem[{Markowitz(1987)}]{Mark1987}
Markowitz HM (1987) Mean-Variance Analysis in Portfolio Choice and Capital
  Markets. Frank J.~Fabozzi Associates, New Hope, Pennsylvania

\bibitem[{Marquardt(1963)}]{Marq63}
Marquardt D (1963) An algorithm for least-squares estimation of nonlinear
  parameters. SIAM Journal on Applied Mathematics 11:431--441

\bibitem[{Morrison(1960)}]{Morr60}
Morrison DD (1960) Methods for nonlinear least squares problems and convergence
  proofs. In: Lorell J, Yagi F (eds) Proceedings of the Seminar on Tracking
  Programs and Orbit Determination, Jet Propulsion Laboratory, Pasadena, USA,
  pp 1--9

\bibitem[{Rockafellar(1970)}]{roc70}
Rockafellar RT (1970) Convex Analysis. Princeton University Press, Princeton

\bibitem[{Rockafellar and Wets(1998)}]{rtrw:1998}
Rockafellar RT, Wets RJB (1998) Variational Analysis, vol 317. Springer, 3rd
  printing

\bibitem[{Waldspurger et~al.(2015)Waldspurger, d'Aspremont, and
  Mallat}]{waldspurger2015phase}
Waldspurger I, d'Aspremont A, Mallat S (2015) Phase recovery, maxcut and
  complex semidefinite programming. Math Program 149(1-2):47--81

\bibitem[{Wiegert(2010)}]{Wiegert:2010}
Wiegert J (2010) The sagacity of circles: A history of the isoperimetric
  problem.
  \url{https://www.maa.org/press/periodicals/convergence/the-sagacity-of-circles-a-history-of-the-isoperimetric-problem},
  accessed: 2018-07-25

\end{thebibliography}
\bibliographystyle{spbasic}      %

\end{document}